\documentclass[a4paper, 10pt]{amsart}
\usepackage{amssymb,amsthm,amstext,amsmath,amscd,latexsym,amsfonts,amscd}
\usepackage{color}
\newtheorem{theorem}{Theorem}[section]
\newtheorem{lemma}[theorem]{Lemma}
\newtheorem{proposition}[theorem]{Proposition}
\newtheorem{corollary}[theorem]{Corollary}

\theoremstyle{definition}
\newtheorem{definition}[theorem]{Definition}

\newtheorem*{Question}{Question}

\theoremstyle{remark}
\newtheorem{remark}[theorem]{Remark}

\newcommand{\Fb}{\mathbb F}

\newcommand{\C}{\mathcal C}
\newcommand{\X}{\mathcal X}
\newcommand{\R}{\mathcal R}

\newcommand{\fa}{\mathfrak a}
\newcommand{\fb}{\mathfrak b}
\newcommand{\m}{\mathfrak m}
\newcommand{\p}{\mathfrak p}
\newcommand{\q}{\mathfrak q}

\newcommand{\syz}{\Omega}

\newcommand{\add}{\mathrm {add}}
\newcommand{\Spec}{\mathrm {Spec}}
\newcommand{\Supp}{\mathrm {Supp}}
\newcommand{\Min}{\mathrm {Min}}
\newcommand{\height}{\mathrm {ht}}
\newcommand{\grade}{\mathrm {grade}}
\newcommand{\depth}{\mathrm {depth}}
\newcommand{\CMd}{\mathrm {CM}\text{-}\mathrm{dim}}

\newcommand{\Rmod}{R\text{-}\mathrm{mod}}

\newcommand{\rmid}[1]{\mathrel{}\middle#1\mathrel{}}

\begin{document}
\title[Annihilators of local cohomology modules]{Annihilators of local cohomology modules via a classification theorem of the dominant resolving subcategories} 

\author{Takeshi Yoshizawa} 
\address{National Institute of Technology, Toyota College, 2-1 Eiseicho, Toyota, Aichi, Japan, 471-8525 }
\email{tyoshiza@toyota-ct.ac.jp}
%
%
%
%
%
%
\begin{abstract}
This paper investigates when local cohomology modules have an annihilator that does not depend on the choice of an ideal. 
Takahashi classified the dominant resolving subcategories of the category of finitely generated modules over a commutative Noetherian ring. 
We show that his classification theorem describes annihilation results of local cohomology modules over a finite-dimensional ring with certain assumptions or a Cohen-Macaulay ring. 
\end{abstract}
%
%
%
\maketitle
%
%
%
\section{Introduction}
Throughout this paper, let $R$ be a commutative Noetherian ring. 

The vanishing of local cohomology modules has been widely studied in local cohomology theory. 
Grothendieck's vanishing theorem \cite[Theorem 6.1.2]{BS-1998} and the Lichtenbaum-Hartshorne vanishing theorem \cite[Theorem 8.2.1]{BS-1998} are perfect examples of indispensable vanishing results. 

Vanishing phenomena of local cohomology modules can also be seen with the help of other factors. 
In one example, Faltings' annihilator theorem \cite{F-1978} states that some power of the second ideal annihilates local cohomology modules over a homomorphic image of a regular ring; see also \cite[Theorem 9.5.1]{BS-1998}.  
In other examples, Raghavan established the uniform annihilation theorem over a homomorphic image $R$ of a biequidimensional regular ring with a finite dimension; see \cite[Theorem 3.1]{R-1992}. 
His theorem guarantees that there exists an integer $n$ (depending only on a finitely generated $R$-module $M$) such that $\fb^n H^{i}_{\fa}(M)=0$ for all ideals $\fa$, $\fb$ of $R$ and all integers $i<\lambda^{\fb}_{\fa}(M)=\inf\left\{ \depth\, M_{\p}+\height(\fa+\p)/\p \mid \p \in \Spec(R)\setminus V(\fb) \right\}$. 
In particular, the ideal $\fb^n$ has some annihilator $s$ such that $s H_{\fa}^{i}(M)=0$ for all ideals $\fa$ of $R$ and all integers $i<\lambda^{\fb}_{\fa}(M)$. 

In 2006, Zhou showed that, if a locally equidimensional ring $R$ of finite positive dimension is a homomorphic image of a Cohen-Macaulay ring of finite dimension or an excellent local ring, then $R$ has a uniform local cohomological annihilator; see \cite[Corollary 3.3]{Z-2006}. 
A uniform local cohomological annihilator is an element $s \in R \setminus (\cup_{\p \in \Min(R)}\, \p)$ such that, for each maximal ideal $\m$ of $R$, one has $sH^{i}_{\m}(R)=0$ for all integers $i < \height\, \m$. 
It should be noted that these annihilators do not depend on the choice of the maximal ideal.

The purpose of this paper is to study vanishing phenomena of local cohomology modules. 
We start by focusing on \cite[Lemma 9.4.3]{BS-1998}, which relates the difficult part of a proof of Faltings' annihilator theorem; see \cite[Chapter 9.4]{BS-1998}. 
This lemma is concerned with the fact that, if $M_{\p}$ is a non-zero free $R_{\p}$-module for a finitely generated $R$-module $M$ and $\p \in \Spec(R)$, then there exists an element $s \in R \setminus \p$ such that $sH^{i}_{\fa}(M)=0$ for all ideals $\fa$ of $R$ and all integers $i <\grade(\fa, R)$. 

In this regard, the paper deals with the following question as to when a similar vanishing phenomenon occurs. 
It goes without saying that the interesting case is when a prime ideal $\p$ belongs to $\Supp_{R}(M)$; see also \cite[Lemma 9.4.1]{BS-1998}.  

\begin{Question}
Let $M$ be a finitely generated $R$-module and let $\p \in \Spec(R)$. 
When does there exist an element $s \in R \setminus \p$ such that $sH^{i}_{\fa}(M)=0$ for all ideals $\fa$ of $R$ and all  integers $i<\grade(\fa, R)$? 
\end{Question}

Our strategy stems from Dao and Takahashi's classification theorem of the dominant resolving subcategories of the category of finitely generated modules over a Cohen-Macaulay ring; see \cite[Theorem 1.4]{DT-2015}. 
Furthermore, it should be noted that Takahashi recently classified such subcategories without the Cohen-Macaulayness of rings; see \cite[Theorem 5.4]{T-2020}.  
These theorems suggest that dominant resolving subcategories have a close relationship with the existence of annihilators of local cohomology modules. 
Indeed, using \cite[Theorem 1.4]{DT-2015}, we provided an answer to the above question in a finite-dimensional Cohen-Macaulay ring; see \cite[Theorem 4.3]{Y-2020}.

One of our aims is to establish the following theorem, which shows that \cite[Theorem 5.4]{T-2020} removes the assumption of finite dimension from \cite[Theorem 4.3]{Y-2020}. 

\begin{theorem}\label{theorem-A}\setlength{\leftmargini}{18pt} 
Suppose that $R$ is a Cohen-Macaulay ring, and let $\p \in \Spec(R)$. 
Then, for a finitely generated $R$-module $M$, the following conditions are equivalent.  
\begin{enumerate}
\item There exists an element $s \in R \setminus \p$ such that $sH_{\fa}^{i}(M)=0$ for all ideals $\fa$ of $R$ and all integers $i < \grade(\fa, R)$. 

\item The $R_{\p}$-module $M_{\p}$ is a maximal Cohen-Macaulay $R_{\p}$-module. 
\end{enumerate}
\end{theorem}

Theorem \ref{theorem-A} leads to the following corollary: the Cohen-Macaulay dimension of a finitely generated module influences whether local cohomology modules have an annihilator that does not depend on the choice of the ideal. 
  
\begin{corollary}\label{corollary-A}
Suppose that $R$ is a Cohen-Macaulay ring, and let $\p \in \Spec(R)$. 
Then, for a finitely generated $R$-module $M$, there exists an element $s \in R \setminus \p$ such that 
\[ sH^{i}_{\fa}(M)=0 \]
for all ideals $\fa$ of $R$ and all integers $i<\grade(\fa, R)-\CMd_{R_{\p}}\, M_{\p}$. 
\end{corollary}

Another purpose is to investigate the above question under (not necessarily Cohen-Macaulay) finite-dimensional rings. 
The following theorem is one of the main results, which provides an answer to our question under some assumptions. 

\begin{theorem}\label{theorem-B}\setlength{\leftmargini}{18pt} 
Suppose that $R$ is a finite-dimensional ring, and let $\p$ be a prime ideal of $R$ with $\depth\, R_{\q}=\grade(\q, R)$ for all  $\q \in U(\p)=\{ \q \in \Spec(R) \mid \q \subseteqq \p \}$. 
Then, for a finitely generated $R$-module $M$, the following conditions are equivalent. 
\begin{enumerate}
\item  There exists an element $s \in R \setminus \p$ such that $sH_{\fa}^{i}(M)=0$ for all ideals $\fa$ of $R$ and all integers $i < \grade(\fa, R)$. 

\item One has $\sup_{\q \in U(\p)} \{ \depth\, R_{\q} - \depth\, M_{\q}\} \leqq 0$.  
\end{enumerate}
\end{theorem}

The above assumption of prime ideals is not an unacceptable condition. 
Indeed, our assumption is very similar to the Cohen-Macaulayness of rings: \cite[Lemma 3.1]{CFF-2002} states that one has $\depth\, R_{\q}=\grade (\q, R)$ for all $\q \in \Spec(R)$ if and only if $R$ is an almost Cohen-Macaulay ring; for details, see also Remark \ref{remark-FD} (2). 

Finally, we can investigate all finitely generated modules by combining Theorem \ref{theorem-B} with the following result. 

\begin{corollary}\label{corollary-B}
Let $R$ and $\p$ be as in Theorem \ref{theorem-B}, and let $M$ be a finitely generated $R$-module with $\sup_{\q \in U(\p)}\{ \depth\, R_{\q} - \depth\, M_{\q}\} \geqq 0$. 
Then there exists an element $s \in R \setminus \p$ such that 
\[ sH^{i}_{\fa}(M)=0 \]
for all ideals $\fa$ of $R$ and all integers $i<\grade(\fa, R) - \sup_{\q \in U(\p)}\{ \depth\, R_{\q} - \depth\, M_{\q}\}$. 
\end{corollary}

\vspace{5pt}
The organization of this paper is as follows. 
We dedicate section \ref{Preliminaries} to the preparation: notations, definitions, and fundamental facts. 
Section \ref{NASC} provides the essential result in this paper; see Theorem \ref{essential-theorem}.  
This result describes that Takahashi's classification theorem \cite[Theorem 5.4]{T-2020} provides a necessary and sufficient condition that local cohomology modules have an annihilator that does not depend on the choice of the ideal. 
After showing Theorem \ref{theorem-B} and Corollary \ref{corollary-B} in section \ref{finite-dimensional}, we establish Theorem \ref{theorem-A} and Corollary \ref{corollary-A} in section \ref{Cohen-Macaulay}.

\section{Preliminaries}\label{Preliminaries}
Throughout this paper, let $R$ be a commutative Noetherian ring, and all modules are assumed to be unitary. 
We denote by $\Rmod$ the category of finitely generated $R$-modules. 
We suppose that all subcategories of $\Rmod$ are full and closed under taking isomorphisms. 
The symbol $\mathbb{N}_0$ denotes the set of non-negative integers. 
We adopt the convention that the grade, the depth, and the dimension of the zero module are $\infty$.

\vspace{5pt}
First of all, we will recall a set of $\mathbb{N}_{0}$-valued functions on $\Spec(R)$ that was introduced by Takahashi in \cite[Definition 5.1 (1)]{T-2020}. 
We note that the symbol $\mathbb{N}$ in the paper \cite{T-2020} stands for the set of non-negative integers. 

\begin{definition}\setlength{\leftmargini}{18pt} 
Let $\mathcal{C}$ be a subcategory of $\Rmod$. 
We denote by $\Fb(\C)$ the set of maps $f \colon \Spec(R) \to \mathbb{N}_{0}$ such that, for every $\p \in \Spec(R)$, there exists $E \in \C$ satisfying the following two conditions:
\begin{enumerate}
\item One has $\depth\, R_{\p} - \depth\, E_{\p}=f(\p)$; 

\item One has $\depth\, R_{\q} - \depth\, E_{\q} \leqq f(\q)$ for all $\q \in \Spec(R)$.
\end{enumerate}
\end{definition}

Next, we need to recall the notions of a dominant subcategory and a resolving subcategory of $\Rmod$.  
We denote by $\syz_{R}^n M$ the $n$th syzygy module of an $R$-module $M$. 
In particular, we will write $\syz_{R}^1 M=\syz_{R} M$, which is the kernel of some epimorphism from a projective module to $M$. 
Note that $\syz_{R}^n M$ is uniquely determined up to projective summands. 

\begin{definition} \setlength{\leftmargini}{18pt} 
Let $\X$ be a subcategory of $\Rmod$. 
\begin{enumerate}
\item We denote by $\add\, \X$ the subcategory of $\Rmod$ consisting of the modules isomorphic to direct summands of finite direct sums of copies of modules in $\X$. 

\item We say that $\X$ is {\it dominant} if, for each $\p \in \Spec(R)$, there exists a non-negative integer $n$ such that $\syz_{R_{\p}}^n \kappa(\p) \in \add\, \X_{\p}$, where $\kappa(\p)=R_{\p}/\p R_{\p}$ and $\X_{\p}=\{ X_{\p} \in R_{\p}\text{-}\mathrm{mod} \mid X \in \X \}$. 

\item We say that $\X$ is a {\it resolving subcategory} if $\X$ contains all finitely generated projective modules and is closed under taking direct summands, extensions, and syzygies. 
\end{enumerate}
\end{definition}

Takahashi showed that the set $\Fb(\Rmod)$ classifies the dominant resolving subcategories of $\Rmod$; see \cite [Definition 5.1 and Theorem 5.4]{T-2020}.  

\begin{theorem}\label{T-theorem}\text{\bf (Takahashi)}. 
There exist mutually inverse order-preserving bijections 
\[ 
\left\{\begin{matrix} \text{ dominant resolving subcategories  of } \Rmod \ \end{matrix} \right\} 
\begin{matrix} \overset{\phi}{\longrightarrow} \cr \underset{\psi}{\longleftarrow}  \cr \end{matrix}  
\ \Fb(\Rmod)
\]
where the maps $\phi$ and $\psi$ are given by
\begin{align*}
& \phi(\X)(\p) = \depth\, R_{\p} - \inf_{X \in \X} \{ \depth\, X_{\p} \} \ \text{ for } \p \in \Spec(R), \text{ and } \\
& \psi(f)=\left\{ X \in \Rmod \mid \depth\, R_{\p} - \depth\, X_{\p} \leqq f(\p) \text{ for all } \p \in \Spec(R) \right\}.
\end{align*}
\end{theorem}

Finally, we will give the notion of a subcategory associated with local cohomology modules and summarize some properties. 
We adopt the convention that the local cohomology functor $H^{i}_{\fa}(-)$ is the zero functor for all ideals $\fa$ of $R$ and all negative integers $i$. 

\begin{definition}\setlength{\leftmargini}{18pt} 
\begin{enumerate} 
\item We define the subcategory $\R(\p)$ of $\Rmod$ relative to a prime ideal $\p$ of $R$ by 
\[ \R(\p)=\left\{ M \in \Rmod \rmid| 
\begin{matrix}
\text{There exists } s \in R \setminus \p \text{ such that } sH^{i}_{\fa}(M)=0 \cr 
\text{ for all ideals } \fa \text{ of } R \text{ and all integers } i < \grade(\fa, R) 
\end{matrix} \right\}. \] 

\item We denote by $U(\p)$ the generalization closed subset $\{\q \in \Spec(R) \mid \q \subseteqq \p\}$ of $\Spec(R)$ for $\p \in \Spec(R)$.
\end{enumerate} 
\end{definition}

A finitely generated $R$-module $M$ over an arbitrary ring $R$ is called a {\it maximal Cohen-Macaulay module} if $\depth\, M_{\p} \geqq \dim R_{\p}$ for all $\p \in \Spec(R)$. 
We note that the depth of the zero module is $\infty$, and thus we consider the zero module to be maximal Cohen-Macaulay. 

\begin{proposition}\label{proposition-R} \setlength{\leftmargini}{18pt} 
Let $\p$ be a prime ideal of $R$.  
\begin{enumerate}
\item For each $\q \in U(\p)$, one has $\R(\p) \subseteqq \{ M \in \Rmod \mid \depth\, M_{\q} \geqq \grade(\q, R) \}$.

\item For each $\q \in \Spec(R) \setminus U(\p)$, the $R$-module $R/\q$ is in $\R(\p)$. 

\item The subcategory $\R(\p)$ is a resolving subcategory of $\Rmod$. 

\item We suppose that $R$ is a finite-dimensional ring. Then the subcategory $\R(\p)$ is dominant. 

\item We suppose that $R$ is a Cohen-Macaulay ring. 
Then the subcategory $\R(\p)$ contains all maximal Cohen-Macaulay $R$-modules. 
In particular, the subcategory $\R(\p)$ is dominant. 
\end{enumerate}
\end{proposition}

\begin{proof}
(1)\,  Let $\q \in U(\p)$. 
We suppose that a finitely generated $R$-module $M$ is in the subcategory $\R(\p)$. 
The definition of $\R(\p)$ gives an element $s \in R \setminus \p$ such that $s H^{i}_{\q}(M)=0$ for all integers $i <\grade(\q, R)$. 
The flat base change theorem \cite[Theorem 4.3.2]{BS-1998} yields $(s/1)H^{i}_{\q R_{\q}}(M_{\q})=0$ where $s/1$ is an element of $R_{\q}$.  
Since the assumption of $\q \in U(\p)$ implies that $s \in R \setminus \q$, the element $s/1$ is a unit of $R_{\q}$. 
Therefore, we obtain $H^{i}_{\q R_{\q}}(M_{\q})=0$ for all integers $i<\grade(\q, R)$. 
 
When we suppose that $\q R_{\q} M_{\q}=M_{\q}$, Nakayama's lemma says that $M_{\q}=0$. 
Thus, we have $\depth\, M_{\q}=\infty >\grade(\q, R)$. 
On the other hand, if we suppose that $\q R_{\q} M_{\q} \neq M_{\q}$, then \cite[Theorem 6.2.7]{BS-1998} deduces $\depth\, M_{\q} \geqq \grade(\q, R)$. 

\noindent
(2)\, Our statement is proved by the same argument in the proof of \cite[Proposition 4.2]{Y-2020} without the assumption of the Cohen-Macaulayness for the ring $R$. 

\noindent
(3)\, We have already shown in \cite[Proposition 3.6 (1)]{Y-2020}. 

\noindent
(4)\, Our assertion follows from \cite[Proposition 3.6 (2)]{Y-2020}. 

\noindent
(5)\, The former assertion follows from \cite[Remark 4.4 (2)]{Y-2020}. 
Furthermore, since $R$ is a Cohen-Macaulay ring, we deduce from \cite[Corollary 4.7]{T-2020} that the resolving subcategory $\R(\p)$ is dominant. 
(We note that the assertion of \cite[Corollary 4.7]{T-2020} has been proved without assuming that the base ring has a finite dimension.)
\end{proof}

\section{Conditions for the existence of an annihilator of local cohomology modules}\label{NASC}
\cite[Theorem 5.4]{T-2020} implies that, if the resolving subcategory $\R(\p)$ for $\p \in \Spec(R)$ is dominant, then there exists the map $f_{\R(\p)} \in \Fb(\Rmod)$ corresponding to $\R(\p)$. 
In this section, we will investigate the relationship between values of the map $f_{\R(\p)}$ and the existence of annihilators of local cohomology modules.

First of all, in the case when the resolving subcategory $\R(\p)$ is dominant, the properties in Proposition \ref{proposition-R} (1) and (2) determine a range of value of the map $f_{\R(\p)}$.  

\begin{lemma}\label{lemma-range} \setlength{\leftmargini}{18pt} 
Let $\q$ be a prime ideal of $R$. 
Suppose that a dominant resolving subcategory $\X$ of $\Rmod$ corresponds to a map $f \in \Fb(\Rmod)$ that is given by Theorem \ref{T-theorem}. 
Then the following assertions hold. 
\begin{enumerate}
\item Let $n(\q)$ be an integer with $0 \leqq n(\q) \leqq \depth\, R_{\q}$. 
If $\X$ is contained in $\{ M \in \Rmod \mid \depth\, M_{\q} \geqq n(\q) \}$, then one has the inequalities $0 \leqq f(\q) \leqq \depth\, R_{\q} - n(\q)$. 

\item If the $R$-module $R/\q$ is in $\X$, then one has the equality $f(\q) = \depth\, R_{\q}$.   
\end{enumerate}
\end{lemma}

\begin{proof}
\cite[Theorem 5.4]{T-2020} states that the subcategory $\X$ can be described as follows: 
\[\X=\left\{ M \in \Rmod \mid \depth\, R_{\p} - \depth\, M_{\p} \leqq f(\p) \text{ for all } \p \in \Spec(R) \right\}.  \eqno{\cdots (*)} \] 

\noindent
(1)\, We note that the map $f \in \Fb(\Rmod)$  is an $\mathbb{N}_{0}$-valued map. 
Therefore, we need to establish the inequality $f(\q)\leqq \depth\, R_{\q} - n(\q)$. 

By the definition of $\Fb(\Rmod)$, for the prime ideal $\q$, there exists a finitely generated $R$-module $E(\q)$ satisfying 
\[ \depth\, R_{\q} - \depth\, E(\q)_{\q} = f(\q), \text{ and } \]
\[\depth\, R_{\p} - \depth\, E(\q)_{\p} \leqq f(\p) \text{ for all } \p \in \Spec(R). \]  

The above equality $(*)$ implies that the module $E(\q)$ is in the subcategory $\X$. 
Therefore, our assumption yields the inequality $\depth\, E(\q)_{\q} \geqq n(\q)$. 
Consequently, we achieve the equality and the inequality
\[ f(\q)=\depth\, R_{\q} - \depth\, E(\q)_{\q} \leqq \depth\, R_{\q} - n(\q). \]

\noindent
(2)\, Since our assumption says that the module $R/\q$ is in the subcategory $\X$, the above equality $(*)$ yields $\depth\, R_{\q} - \depth\, (R/\q)_{\q} \leqq f(\q)$. 
On the other hand, the definition of $\Fb(\Rmod)$ gives the equality $\depth\, R_{\q} -\depth\, E(\q)_{\q}=f(\q)$ for some finitely generated $R$-module $E(\q)$. 
Consequently, we have
\begin{align*}
\depth\, R_{\q}=\depth\, R_{\q} -\depth\, \kappa\left(\q \right) \leqq f(\q) \leqq \depth\, R_{\q}, 
\end{align*}
that is the equality $f(\q)=\depth\, R_{\q}$. 
\end{proof}

\begin{remark} 
Let $\q$ be a prime ideal of $R$.  

\noindent
(1)\, We note that each map $f \in \Fb(\Rmod)$ satisfies inequalities $0 \leqq f(\p) \leqq \depth\, R_{\p}$ for all $\p \in \Spec(R)$. 
Considering this fact, the assumption of inequalities $0 \leqq n(\q) \leqq \depth\, R_{\q}$ is appropriate to investigate the assertion (1) in Lemma \ref{lemma-range}. 

\noindent
(2)\, \cite[Proposition 1.2.10 (a)]{BH-1998} states that one has 
\[ 0 \leqq \grade(\q, R) = \inf\{ \depth\, R_{\p} \mid \p \in V(\q) \} \leqq \depth\, R_{\q}.\]  
\end{remark}

Next, Lemma \ref{lemma-range} (2) makes the equality (*) in the proof for Lemma \ref{lemma-range} a slightly simpler form. 

\begin{lemma}\label{lemma-basic}
We suppose that a resolving subcategory $\X$ of $\Rmod$ is dominant. 
By Theorem \ref{T-theorem}, the subcategory $\X$ corresponds to a map $f \in \Fb(\Rmod)$. 
Let $S$ be a subset of the set $\{\q \in \Spec(R) \mid R/\q \in \X \}$. 
Then one has the following equality of subcategories of $\Rmod$: 
\[ \X = \{ M \in \Rmod \mid \depth\, R_{\q} - \depth\, M_{\q} \leqq f(\q) \text{ for all } \q \in \Spec(R) \setminus S  \}. \]
\end{lemma}

\begin{proof}
Let $\q$ be a prime ideal in the set $S$, and let $M$ be a finitely generated $R$-module. 
Since the $R$-module $R/\q$ is in the subcategory $\X$, Lemma \ref{lemma-range} (2) implies that  
\[ \depth\, R_{\q} - \depth\, M_{\q} \leqq \depth\, R_{\q} = f(\q). \] 
Consequently, we can establish the following equalities of subcategories of $\Rmod$ 
\begin{align*}
\X &=\left\{ M \in \Rmod \mid \depth\, R_{\q} - \depth\, M_{\q} \leqq f(\q) \text{ for all } \q \in \Spec(R) \right\} & \\ 
&=\left\{ M \in \Rmod \mid \depth\, R_{\q} - \depth\, M_{\q} \leqq f(\q) \text{ for all } \q \in \Spec(R) \setminus S \right\} 
\end{align*}
by \cite[Theorem 5.4]{T-2020}. 
\end{proof}

Let $\p$ be a prime ideal of $R$. 
According to Proposition \ref{proposition-R} (2), the subcategory $\R(\p)$ has the relation $\Spec(R) \setminus U(\p) \subseteqq \{ \q \in \Spec(R) \mid R/\q \in \R(\p)\}$. 
Consequently, as an immediate consequence of Lemma \ref{lemma-basic}, we now present the essential result in this paper.  
The following theorem provides a necessary and sufficient condition that local cohomology modules have a {\it dominant annihilator}, in the sense that this annihilator does not depend on the choice of the ideal for local cohomology modules. 

\begin{theorem}\label{essential-theorem}\setlength{\leftmargini}{18pt} 
Let $\p$ be a prime ideal of $R$. 
Suppose that the resolving subcategory $\R(\p)$ is dominant, and also that $\R(\p)$ corresponds to
the map $f_{\R(\p)}\in \Fb(\Rmod)$ that is given by Theorem \ref{T-theorem}. 
Then, for a finitely generated $R$-module $M$, the following conditions are equivalent. 
\begin{enumerate}
\item  There exists an element $s \in R \setminus \p$ such that $sH_{\fa}^{i}(M)=0$ for all ideals $\fa$ of $R$ and all integers $i < \grade(\fa, R)$. 

\item One has $\depth\, R_{\q}-\depth\, M_{\q} \leqq f_{\R(\p)}(\q)$ for all $\q \in U(\p)$. 
\end{enumerate}
\end{theorem}

Theorem \ref{essential-theorem} reveals that Takahashi's classification theorem \cite[Theorem 5.4]{T-2020} suggests the following conclusion: the questions below have a closed relationship with the existence of a dominant annihilator of local cohomology modules. 

\begin{Question}\setlength{\leftmargini}{18pt} 
Let $\p$ be a prime ideal of $R$. 
\begin{enumerate}
\item When is the resolving subcategory $\R(\p)$ dominant? 

\item What value does $f_{\R(\p)}(\q)$ take for each $\q \in U(\p)$? 
\end{enumerate}
\end{Question}

Proposition \ref{proposition-R} (4) (respectively, (5)) states that the above question (1) has an affirmative answer when $R$ is a finite-dimensional ring (respectively, a Cohen-Macaulay ring). 
Regarding the above question (2), the next section will give an answer when $R$ has a finite dimension under certain assumptions. 
Furthermore, the last section will provide a complete answer when $R$ is a Cohen-Macaulay ring.

\section{The case of finite-dimensional rings}\label{finite-dimensional}
The purpose of this section is to establish annihilation results for local cohomology modules over a finite-dimensional ring with certain assumptions.

In the case when a prime ideal $\p$ of a finite-dimensional ring $R$ has the equality $\depth\, R_{\q}=\grade(\q, R)$ for all $\q \in U(\p)$, we can describe a condition for the existence of an annihilator of local cohomology modules without using the notion of maps in $\Fb(\Rmod)$. 

\begin{theorem}\label{theorem-FD}\setlength{\leftmargini}{18pt} 
Let $\p$ be a prime ideal of $R$ with $\depth\, R_{\q}=\grade(\q, R)$ for all  $\q \in U(\p)$. 
Suppose that the resolving subcategory $\R(\p)$ is dominant (e.g., the ring $R$ has a finite dimension). 
Then, for a finitely generated $R$-module $M$, the following conditions are equivalent. 
\begin{enumerate}
\item  There exists an element $s \in R \setminus \p$ such that $sH_{\fa}^{i}(M)=0$ for all ideals $\fa$ of $R$ and all integers $i < \grade(\fa, R)$. 

\item One has $\sup_{\q \in U(\p)} \{ \depth\, R_{\q} - \depth\, M_{\q}\} \leqq 0$.  
\end{enumerate}
\end{theorem}

\begin{proof}
Note that finite-dimensional rings have the dominant resolving subcategory $\R(\p)$ by Proposition \ref{proposition-R} (4). 

Let $R$ be a (not necessarily finite-dimensional) ring such that the resolving subcategory $\R(\p)$ is dominant. 
Then \cite[Theorem 5.4]{T-2020} gives the map $f \in \Fb(\Rmod)$ corresponding to the subcategory $\R(\p)$. 

We will show that the equality $f(\q)=0$ holds for each $\q \in U(\p)$. 
Indeed, Proposition \ref{proposition-R} (1) and our assumption yield the inclusion relation  
\[ \R(\p) \subseteqq \{ M \in \Rmod \mid \depth\, M_{\q} \geqq \depth\, R_{\q} \}.\] 
Thus, we can deduce from Lemma \ref{lemma-range} (1) that $0 \leqq f(\q)=\depth\, R_{\q} - \depth\, R_{\q}=0$. 

Theorem \ref{essential-theorem} establishes the equalities 
\begin{align*}
\R(\p) &=\{M \in \Rmod \mid \depth\, R_{\q} -\depth\, M_{\q} \leqq f(\q) \text{ for all } \q \in U(\p)\} &\\ 
&=\{M \in \Rmod \mid \depth\, R_{\q} -\depth\, M_{\q} \leqq 0 \text{ for all } \q \in U(\p) \}. 
\end{align*} 
These equalities mean that an $R$-module $M$ is in the subcategory $\R(\p)$ if and only if one has $\sup_{\q \in U(\p)} \{ \depth\, R_{\q} - \depth\, M_{\q}\} \leqq 0$. 
\end{proof}

\begin{remark}\label{remark-FD}\setlength{\leftmargini}{18pt} 
Let $\p$ be a prime ideal of $R$.
\begin{enumerate}
\item In the preceding section, we talked about the question for the value of $f_{\R(\p)}(\q)$ for each $\q \in U(\p)$. 
As in the proof for Theorem \ref{theorem-FD}, if $R$ is a finite-dimensional ring and $\p$ satisfies the equality  $\depth\, R_{\q}=\grade(\q, R)$ for all $\q \in U(\p)$, then one has the value $f_{\R(\p)}(\q)=0$ for each $\q \in U(\p)$.  

\item A ring $R$ is said to be {\it almost Cohen-Macaulay} if it has the {\it Cohen-Macaulay defect} $\mathrm{cmd}\, R=\sup_{\q \in \Spec(R)}\{ \dim R_{\q} - \depth\, R_{\q} \} \leqq 1$. 
The characterization \cite[Lemma 3.1]{CFF-2002} states that a ring $R$ is almost Cohen-Macaulay if and only if one has $\depth\, R_{\q}=\grade (\q, R)$ for all $\q \in \Spec(R)$. 

\item An ideal $I$ of $R$ is called {\it perfect} if the projective dimension of the $R$-module $R/I$ is equal to $\grade (I, R)$; see \cite[Definition 1.4.15]{BH-1998}. 
In the case when a prime ideal $\q$ of $R$ is perfect, one has $\depth\, R_{\q}=\grade(\q, R)$; see \cite[Proposition 1.4.16]{BH-1998}. 

\item For a finitely generated $R$-module $M$, the {\it large restricted flat dimension} $\mathrm{Rfd}_{R}\, M$ is defined by $\sup_{\q \in \Spec(R)}\{ \depth\, R_{\q} - \depth\, M_{\q} \}$. 
The following inequalities are an immediate consequence of \cite[Proposition 1.2.12]{BH-1998}:  
\[\sup_{\q \in U(\p)}\{ \depth\, R_{\q} - \depth\, M_{\q}\} \leqq \min\{\dim R_{\p}, \mathrm{Rfd}_{R}\, M \} \leqq \dim R. \]
See also \cite[Proposition 2.2 and Theorem 2.4]{CFF-2002}. 
\end{enumerate}
\end{remark}

Theorem \ref{theorem-FD} concludes that the discussion now turns to the existence of annihilator of local cohomology modules for a finitely generated $R$-module $M$ with $\sup_{\q \in U(\p)}\{ \depth\, R_{\q} -\depth\, M_{\q}\}> 0$. 
Theorem \ref{theorem-FD} describes the following corollary as an answer to the discussion about such modules. 

\begin{corollary}\label{corollary-FD}
Let $\p$ be a prime ideal of $R$ with $\depth\, R_{\q}=\grade(\q, R)$ for all  $\q \in U(\p)$. 
Suppose that the resolving subcategory $\R(\p)$ is dominant (e.g., the ring $R$ has a finite dimension). 
Let $M$ be a finitely generated $R$-module with $\sup_{\q \in U(\p)}\{ \depth\, R_{\q} - \depth\, M_{\q}\} \geqq 0$. 
Then there exists an element $s \in R \setminus \p$ such that 
\[ sH^{i}_{\fa}(M)=0 \]
for all ideals $\fa$ of $R$ and all integers $i<\grade(\fa, R) - \sup_{\q \in U(\p)}\{ \depth\, R_{\q} - \depth\, M_{\q}\}$. 
\end{corollary}

\begin{proof}
We have the inequalities $\sup_{\q \in U(\p)}\{ \depth\, R_{\q} - \depth\, M_{\q}\} \leqq \dim R_{\p}< \infty$ by Remark \ref{remark-FD} (4). 
Let $n=\sup_{\q \in U(\p)}\{ \depth\, R_{\q} - \depth\, M_{\q}\}$. 
Our assumption guarantees that $n$ is a non-negative integer. 

In the case when $n=0$, our assertion is immediately from Theorem \ref{theorem-FD}. 

Now suppose that $n \geqq 1$. 
Since each prime ideal $\q$ in $U(\p)$ has $\depth\, R_{\q} \leqq \depth\, M_{\q}+n$, the depth lemma \cite[Proposition 1.2.9]{BH-1998} yields the inequalities and the equality 
\begin{align*}
\depth \left( (\syz^{n}_{R} M)_{\q} \right) 
&\geqq \min\{\depth\, R_{\q}, \depth \left(\syz_{R}^{n-1} M\right)_{\q}+1\} 
\geqq \cdots &\\
&\geqq \min\{\depth\, R_{\q}, \depth\, M_{\q}+n \}
= \depth\, R_{\q}. 
\end{align*}
Therefore, one has the inequality $\sup_{\q \in U(\p)}\{ \depth\, R_{\q} - \depth \left( (\syz^{n}_{R} M)_{\q} \right) \} \leqq 0$. 
We now apply Theorem \ref{theorem-FD} to the module $\syz^{n}_{R} M$. 
Then there exists an element $s \in R \setminus \p$ such that $s H^{i}_{\fa}(\syz_{R}^{n} M)=0$ for all ideals $\fa$ of $R$ and all integers $i<\grade(\fa, R)$. 

Our aim is to show that the above element $s$ yields $sH^{i}_{\fa}(M)=0$ for all ideals $\fa$ of $R$ and all integers $i<\grade(\fa, R)-n$. 
To prove our assertion, we now fix an ideal $\fa$ of $R$. 

We suppose that one has $\grade(\fa, R)-n \leqq 0$. 
Since the local cohomology functor $H^{i}_{\fa}(-)$ is the zero functor for all integers $i<0$, we can achieve the equality $sH^{i}_{\fa}(M)=0$ for all integers $i<\grade(\fa, R)-n$. 

Next suppose that we have $\grade(\fa, R)-n \geqq 1$. 
It should be noted that one has the inequality $\grade(\fa, R)\geqq 2$. 
A projective resolution 
\[ \cdots \to P_{k} \to \cdots \to P_{1} \to P_{0} \to M \to 0 \]
of the $R$-module $M$ provides short exact sequences 
\[ H^{i-1}_{\fa}(P_{k-1}) \to H^{i-1}_{\fa}(\syz_{R}^{k-1}M) \to H^{i}_{\fa}(\syz_{R}^{k}M) \to H^{i}_{\fa}(P_{k-1})\]
for all positive integers $i$ and $k$. 

Since the $R$-module $P_{k-1}$ is projective, or a direct summand of a free module, \cite[Theorem 6.2.7]{BS-1998} yields $H^{i}_{\fa}(P_{k-1})=0$ for all integers $i<\grade(\fa, R)$. 
Consequently, the above element $s$ provides that 
\[ sH^{i}_{\fa}(M) \cong sH^{i+1}_{\fa}(\syz_{R} M) \cong sH^{i+n}_{\fa}(\syz^{n}_{R} M)=0\] 
for all integers $i$ with $i+n<\grade(\fa, R)$. 
The proof of our corollary is completed. 
\end{proof}

\section{The case of Cohen-Macaulay rings}\label{Cohen-Macaulay}
This section will investigate the existence of annihilators of local cohomology modules over a Cohen-Macaulay ring. 
In particular, our purpose is to verify that Cohen-Macaulay rings establish results similar to Theorem \ref{theorem-FD} and Corollary \ref{corollary-FD}.

We begin with the following easy lemma. 

\begin{lemma}\label{lemma-CM}
We suppose that $R$ is a Cohen-Macaulay ring. 
Let $\p \in \Spec(R)$, and let $M$ be a finitely generated $R$-module with $\depth\, M_{\p} \geqq \depth\,  R_{\p}$. 
Then one has $\depth\, M_{\q} \geqq \depth\, R_{\q}$ for all $\q \in U(\p)$. 
\end{lemma}

\begin{proof}
Let $\q$ be a prime ideal of $R$ with $\q \subseteqq \p$.  
Since $R$ is a Cohen-Macaulay ring, the module $M$ satisfies the inequalities and the equalities 
\begin{align*}
\depth\, M_{\q} 
&\geqq \depth\, M_{\p}-\height\, \p /\q 
\geqq \depth\, R_{\p}-\height\, \p /\q &\\ 
&=\dim R_{\p} -(\height\p-\height \q)
=\dim R_{\q}=\depth\, R_{\q} 
\end{align*}
by \cite[Lemma 9.3.2]{BS-1998} or \cite[Lemma 6.2]{T-2020}.  
\end{proof}

We now achieve the following conclusion about when local cohomology modules over a Cohen-Macaulay ring have an annihilator that does not depend on the choice of the ideal. 
The result below removes from \cite[Theorem 4.3]{Y-2020} the assumption that the base ring has a finite dimension.  

\begin{theorem}\label{theorem-CM}\setlength{\leftmargini}{18pt}
Suppose that $R$ is a Cohen-Macaulay ring, and let $\p \in \Spec(R)$. 
Then, for a finitely generated $R$-module $M$, the following conditions are equivalent.  
\begin{enumerate}
\item  There exists an element $s \in R \setminus \p$ such that $sH_{\fa}^{i}(M)=0$ for all ideals $\fa$ of $R$ and all integers $i < \grade(\fa, R)$.   

\item The $R_{\p}$-module $M_{\p}$ is a maximal Cohen-Macaulay $R_{\p}$-module. 
\end{enumerate}
\end{theorem}

\begin{proof}
Note that the Cohen-Macaulay ring $R$ has the dominant resolving subcategory $\R(\p)$ by Proposition \ref{proposition-R} (5). 
Hence, \cite[Theorem 5.4]{T-2020} gives the map $f \in \Fb(\Rmod)$ corresponding to the subcategory $\R(\p)$. 

We claim that the equality $f(\q)=0$ holds for each $\q \in U(\p)$. 
Indeed, since $R$ is a Cohen-Macaulay ring, \cite[Theorem 2.1.3]{BH-1998} yields the equality $\depth\,  R_{\q}=\grade(\q, R)$.
It therefore follows from Proposition \ref{proposition-R} (1) that 
\[ \R(\p) \subseteqq \{ M \in \Rmod \mid \depth\, M_{\q} \geqq \depth\, R_{\q} \}.\] 
Applying Lemma \ref{lemma-range}, we obtain $0 \leqq f(\q)=\depth\, R_{\q} - \depth\, R_{\q}=0$. 

We now establish the equalities 
\begin{align*}
\R(\p)&=\{M \in \Rmod \mid \depth\, R_{\q} -\depth\, M_{\q} \leqq f(\q) \text{ for all } \q \in U(\p)\} &\\ 
&=\{M \in \Rmod \mid \depth\, R_{\q} -\depth\, M_{\q} \leqq 0 \text{ for all } \q \in U(\p)\} \\
&=\{ M \in \Rmod \mid \depth\, M_{\p} \geqq \depth\, R_{\p}\} \\
&=\{ M \in \Rmod \mid \depth\, M_{\p} \geqq \dim R_{\p} \}
\end{align*} 
using Theorem \ref{essential-theorem}, the above claim, Lemma \ref{lemma-CM}, and the Cohen-Macaulayness for the ring $R$, respectively. 
Consequently, we can conclude that an $R$-module $M$ is in the subcategory $\R(\p)$ if and only if an $R_{\p}$-module $M_{\p}$ is a maximal Cohen-Macaulay $R_{\p}$-module. 
\end{proof}

\begin{remark}\label{remark-CM}\setlength{\leftmargini}{18pt}
Let  $\p$ be a prime ideal of $R$.  
\begin{enumerate}
\item When $R$ is a Cohen-Macaulay ring, the proof for Theorem \ref{theorem-CM} presents the complete answer to the question (2) in section \ref{NASC}.
Indeed, with the notation of Theorem \ref{essential-theorem}, we have already established the equality $f_{\R(\p)}(\q)=0$ for each $\q \in U(\p)$. 

\item We can regard Theorem \ref{theorem-CM} as an immediate consequence of Theorem \ref{theorem-FD}. 
To prove this fact, we now suppose that $R$ is a Cohen-Macaulay ring. 
By \cite[Theorems 3.8, 3.9, and Proposition 3.10]{G-2001}, a finitely generated $R$-module $M$ has the equalities 
\begin{align*}
\sup_{\q \in U(\p)}\{ \depth\, R_{\q} - \depth\, M_{\q} \} 
&= \sup_{\q \in U(\p)} \{ \CMd_{R_{\q}}\, M_{\q} \} \\
&=\CMd_{R_{\p}}\, M_{\p} \\
&=\depth\, R_{\p} - \depth\, M_{\p} \\
&=\dim R_{\p} - \depth\, M_{\p}.  
\end{align*}
Note that, for each $\q \in \Spec(R)$, one has $\depth\, M_{\q}=\infty$ if and only if $M_{\q}=0$ if and only if $\CMd_{R_{\q}}\, M_{\q}=-\infty$.  
\end{enumerate}
\end{remark}

The corollary below is a generalization of \cite[Corollary 4.6]{Y-2020}, which guarantees the existence of annihilators of local cohomology modules for finitely generated modules over a finite-dimensional Cohen-Macaulay ring. 
Using Theorem \ref{theorem-CM}, the same proof for \cite[Corollary 4.6]{Y-2020} works without the assumption that the base ring has a finite dimension. 

\begin{corollary}\label{corollary-CM}
Suppose that $R$ is a Cohen-Macaulay ring, and let $\p \in \Spec(R)$. 
Then, for a finitely generated $R$-module $M$, there exists an element $s \in R \setminus \p$ such that 
\[ sH^{i}_{\fa}(M)=0 \]
for all ideals $\fa$ of $R$ and all integers $i<\grade(\fa, R)-\CMd_{R_{\p}}\, M_{\p}$. 
\end{corollary}

\begin{remark}\setlength{\leftmargini}{18pt}
Let $R$ be a Cohen-Macaulay ring, let $\p$ be a prime ideal of $R$, and let $M$ be a finitely generated $R$-module.
\begin{enumerate}
\item We can regard Corollary \ref{corollary-CM} as an immediate consequence of Corollary \ref{corollary-FD} by the equalities in Remark \ref{remark-CM} (2). 

Note that, in the case when $\CMd_{R_{\p}}\, M_{\p}< 0$, we have the equality  $\CMd_{R_{\p}}\, M_{\p}=-\infty$. 
Namely, one has $M_{\p}=0$. 
It is easy to see that there exists an element $s \in R \setminus \p$ such that $sM=0$. 
Since local cohomology functors are $R$-linear by \cite[Properties 1.2.2]{BS-1998}, we achieve the equalities $sH^{i}_{\fa}(M)=0$ for all ideals $\fa$ of $R$ and all integers $i$; see also \cite[Lemma 9.4.1]{BS-1998}. 

\item Let $t$ be a non-negative integer. 
We suppose that $M$ has an element $s \in R \setminus \p$ such that $sH^{i}_{\fa}(M)=0$ for all ideals $\fa$ of $R$ and all integers $i<\grade(\fa, R)-t$. 
Then the following investigation will lead to the inequality $t \geqq \CMd_{R_{\p}} M_{\p}$. 

When the equality $M_{\p}=0$ holds, one has $\CMd_{R_{\p}}\, M_{\p}=-\infty$. Therefore, there is nothing to prove.  

Next suppose that $M_{\p}\neq 0$. 
We now take the prime ideal $\p$ as the above ideal $\fa$. 
The flat base change theorem \cite[Theorem 4.3.2]{BS-1998} yields $H^{i}_{\p R_{\p}}(M_{\p})=0$ for all integers $i <\grade(\p, R)-t$. 
Since Nakayama's lemma guarantees that $\p R_{\p}M_{\p} \neq M_{\p}$, we deduce from \cite[Theorem 6.2.7]{BS-1998} and \cite[Theorem 2.1.3]{BH-1998} that 
\[ \depth\, M_{\p} \geqq \grade(\p, R) - t = \depth\, R_{\p} - t.\] 

Consequently, we can conclude $t \geqq \depth\, R_{\p} - \depth\, M_{\p} = \CMd_{R_{\p}}\, M_{\p}$ by \cite[Theorems 3.8 and 3.9]{G-2001}. 
\end{enumerate}
\end{remark}

\section*{Acknowledgments}
This work was supported by JSPS KAKENHI Grant Number JP20K03549.

\vspace{7pt}
\bibliographystyle{amsplain}

\begin{thebibliography}{99}
\bibitem{BS-1998}
M. P. Brodmann; R. Y. Sharp, 
Local cohomology: an algebraic introduction with geometric applications. 
Cambridge Studies in Advanced Mathematics, \textbf{60}, 
{\it Cambridge University Press, Cambridge}, 1998. 

\bibitem{BH-1998}
W. Bruns; J. Herzog,
Cohen-Macaulay rings, revised edition. 
Cambridge Studies in Advanced Mathematics, \textbf{39}, 
{\it Cambridge University Press, Cambridge}, 1998.

\bibitem{CFF-2002}
L. W. Christensen; H.-B. Foxby; A. Frankild,  
Restricted homological dimensions and Cohen-Macaulayness. 
{\it J. Algebra} \textbf{251} (2002),  no. 1, 479--502.

\bibitem{DT-2015}
H. Dao; R. Takahashi, 
Classification of resolving subcategories and grade consistent functions. 
{\it Int. Math. Res. Not. IMRN} \textbf{2015} (2015),  no. 1, 119--149.

\bibitem {F-1978}
G. Faltings, 
$\ddot{\text{U}}$ber die Annulatoren lokaler Kohomologiegruppen. 
{\it Arch. Math. (Basel)} \textbf{30} (1978), no. 5, 473--476. 

\bibitem {G-2001}
A. A. Gerko, 
On homological dimensions. 
{\it Mat. Sb.} \textbf{192} (2001), no. 8, 79--94; 
translation in {\it Sb. Math.} \textbf{192} (2001), no. 7--8, 1165--1179. 

\bibitem{R-1992}
K. Raghavan, 
Uniform annihilation of local cohomology and of Koszul homology. 
{\it Math. Proc. Cambridge Philos. Soc.} \textbf{112} (1992), no. 3, 487--494. 

\bibitem{T-2020}
R. Takahashi, 
Classification of dominant resolving subcategories by moderate functions. 
To appear in Illinois J. Math., https://www.math.nagoya-u.ac.jp/$\sim$takahashi/papers.html (accessed 2021-4-30).


\bibitem{Y-2020}
T. Yoshizawa, 
Annihilators of local cohomology modules over a Cohen-Macaulay ring. 
Submitted, 2020. 

\bibitem{Z-2006}
C. Zhou, 
Uniform annihilators of local cohomology. 
{\it J. Algebra} \textbf{305} (2006), no. 1, 585--602.
\end{thebibliography}

\end{document}